\documentclass[10pt]{amsart}
\usepackage{amssymb}
\usepackage{bm}
\usepackage[centertags]{amsmath}
\usepackage{amsfonts}
\usepackage{amsthm}
\usepackage{mathtools}
\usepackage[numbers, square]{natbib}
\usepackage{hyperref}
\newtheorem{thm}{Theorem}[section]

\newtheorem{lem}[thm]{Lemma}
\newtheorem{prop}[thm]{Proposition}

\theoremstyle{definition}
\newtheorem{defn}[thm]{Definition}

\DeclarePairedDelimiter\abs{\lvert}{\rvert}
\DeclarePairedDelimiter\norm{\lVert}{\rVert}
\DeclarePairedDelimiterX\Set[1]\{\}{%

#1
}

\begin{document}
\title{Simplices and Regular Polygonal Tori in Euclidean Ramsey Theory}

\author{Miltiadis Karamanlis
}
\thanks{National Technical University of Athens, Faculty of Applied Sciences,
Department of Mathematics, Zografou Campus, 157 80, Athens, Greece. Email: 
\href{kararemilt@gmail.com} {\nolinkurl{kararemilt@gmail.com}}. 
Research supported in part by E.L.K.E. of N.T.U.A.}

%------------------------Abstract-------------------------------%
\begin{abstract} We show that any finite affinely independent set can be isometrically embedded into a regular polygonal torus, that is, a finite product of regular polygons. As a consequence, with a straightforward application of K\v{r}\'{i}\v{z}'s theorem, we get an alternative proof of the fact that all finite affinely independent sets are Ramsey, a result which was originally proved by Frankl and R\"{o}dl.
\end{abstract}

%--------------------------Introduction-----------------------%

\maketitle
\numberwithin{equation}{section}

\section{Introduction}
Let us start by recalling some basic concepts and classical results of Euclidean Ramsey theory which form the context of the work presented in this note. A finite set $X\subset\mathbb{R}^n$ is \textit{Ramsey} if for every $r\in\mathbb{N}$ there exists $N=N(X,r)\in\mathbb{N}$ such that for any $r$-coloring of $\mathbb{R}^N$ there exists a monochromatic isometric copy $X'\subset\mathbb{R}^N$ of $X$. Ramsey sets  where first introduced and studied by Erd\H{o}s,  Graham, Montgomery, Rothschild,  Spencer, and Straus in  \citep{ErGrMoRSS1973}. There, among others, they proved that Cartesian products of Ramsey sets are Ramsey and that every Ramsey set is spherical, that is, it lies on the surface of some sphere. It is a famous open conjecture due to Graham \cite{Graham1994} that the converse also holds, that is, every spherical set is Ramsey.

Two of the most significant results in  Euclidean Ramsey theory appeared almost simultaneously around the dawn of 90s. Frankl and R\"{o}dl in \cite{FR90} proved that every simplex, that is, any finite set of affinely independent points, is Ramsey. One year later, K\v{r}\'\i\v{z} in \cite{Kriz1991} proved that any finite set with a transitive\footnote{For $X\subseteq\mathbb{R}^n$, a group $G$ of isometries of $X$ is called transitive if for every $x,x'\in X$ there exists $g\in G$ such that $gx=x'$. Sets with a transitive group of isometries will be also called transitive.}  solvable group of isometries is Ramsey. In particular, all regular polygons are Ramsey. 

Frankl and R\"{o}dl in \cite{FR90}, actually showed that all simplices are exponentially Ramsey, that is, for any simplex $X$, there exists $\varepsilon=\varepsilon(X)>0$, such that, every coloring of $\mathbb{R}^n$ with fewer than $(1+\varepsilon)^n$ colors contains a monochromatic copy of $X$ (for further results on the Ramsey properties of simplices see  \cite{MR95} and \cite{Frankl2004}). On the other hand, although K\v{r}\'\i\v{z}'s theorem doesn't provide much quantitative information, it seems to be closely connected with the problem of characterizing the Ramsey sets.

Let us mention here the more recent conjecture of Leader, Russell and Walters \cite{LeaRusWal2012} that Ramsey sets are, up to isometry, exactly the subsets of the finite transitive sets (their conjecture differs from  Graham's since this class of sets is strictly contained in that of spherical sets  \cite{LeaRusWal2011}, \cite{LeaRusWal2012}, \cite{Eberhard2013}). The starting point of their conjecture was the observation that all known Ramsey sets  embed\footnote{Throughout this note, we say that a set $X\subset\mathbb{R}^n$ \emph{embeds into} a set $Y\subset\mathbb{R}^m$, if there exists $f:X\to Y$, such that $\|f(x)-f(x')\|=\|x-x'\|$,  for every $x,x'\in X$ (where $\|\cdot\|$ denotes the usual Euclidean norm).} into some transitive set. In the simple case of triangles (or even for some kinds of quadrilaterals, such as the  isosceles trapezoids), there are elementary geometric constructions demonstrating that they can be embedded into a three dimensional (twisted) prism with a transitive solvable group of isometries (see for example  \cite{johnson_2006}, \cite{LeaRusWal2012} for more details).

It was a natural question for us whether this fact has a higher dimensional analogue, namely, whether all simplices embed into finite sets with a transitive solvable group of isometries. In this note, we answer this question positively in the simplest possible way, by using the following generalization of prisms.
\begin{defn} Let $(T_i)_{i=1}^n$ be a finite sequence of regular polygons in $\mathbb{R}^2$. The product
$\displaystyle T=\prod_{i=1}^n T_i=\left\{(t_i)_{i=1}^n \ : t_i\in T_i \ \forall i=1,\dots,n \right\}\subseteq \mathbb{R}^{2n}$
will be called  \emph{regular polygonal torus}.
\end{defn}
We may view the regular polygonal tori as discrete versions of the so called \emph{Clifford tori}, i.e. products of finitely many circles.
Notice that regular polygonal tori have an abelian transitive group of isometries. Moreover, using some elements from Linear Algebra, for example the fact that commuting  unitary transformations admit a simultaneous diagonalization, it follows that every finite set with a transitive abelian group of isometries is actually a subset of a regular polygonal torus.

Our main result is the following.
\begin{thm}\label{thm:Simplices_Into_Tori}
Every simplex embeds into a regular polygonal torus.
\end{thm}
The above theorem provides an alternative proof that simplices are Ramsey via K\v{r}\'\i\v{z}'s theorem, and thus, it creates a link between these two fundamental results. Indeed, by K\v{r}\'\i\v{z}'s theorem, every regular polygonal torus is Ramsey, and since by their definition Ramsey sets are closed under subsets, every set which embeds into a regular polygonal torus is Ramsey.
\section{Notation}
In order to make the statements more precise, we first set up some notation. Let $\mathbb{N}$ be the set of positive integers. For a finite set $X$, by $|X|$ we will denote its cardinality. For $n\in\mathbb{N}$, we set  $[n]=\{1,2,\ldots,n\}$. The $n$-dimensional \emph{Euclidean space} is the vector space $\mathbb{R}^n$ equipped with the usual Euclidean norm $\|\cdot\|$.

For $m\geq 2$ and $r>0$ by $T_{m,r}\subset\mathbb{R}^2$ we  generally denote (the set of vertices of) a regular $m$-gon of circumradius $r$.  Hence, a \emph{regular polygonal torus}  $T$  is a set of the form $T=\prod_{i\in[n]}T_{m_i,r_i}$,  where $(m_i)_{i\in[n]}\in\mathbb{N}^n$ and $(r_i)_{i\in[n]}\in \mathbb{R}^n_{+}$. If  $m_i=m$ for all $i\in[n]$, i.e. $T=\prod_{i\in[n]}T_{m,r_i}$ then $T$ will be called $m$-\emph{regular}. If in addition $r_i=r$ for all $i\in[n]$, then the regular polygonal torus will be denoted as $T=T_{m,r}^n$ and it will be called $(m,r)$-\emph{regular}.

\section{The proof of Theorem \ref{thm:Simplices_Into_Tori}}\label{sec:Properties_of_regular_polygonal_tori}

The proof of Theorem \ref{thm:Simplices_Into_Tori} makes use of a result due to Matou\v{s}ek and R\"{o}dl in \cite{MR95} and shares some common features with the original proof of Frankl and R\"{o}dl in \cite{FR90} that all simplices are Ramsey. The arguments in \cite{FR90} rely on techniques from extremal set theory, including some deep results there such as those found in the work \cite{FR87} of the same authors. Here the proof of the fact that all simplices are Ramsey comes as a corollary of  K\v{r}\'\i\v{z}'s theorem and as a result although it provides worse numerical bounds it is much simpler.

An outline of our proof is as follows. Let us denote by $\mathcal{T}$ the class of all subsets of regular polygonal tori. The first step is to show that $\mathcal{T}$ contains (up to isometry) all regular simplices, as well as small perturbations of them. The second step is to show that $\mathcal{T}$ is ``dense'', in the sense that it contains almost isometric copies of every finite set. In the third step we show that every \emph{regular expansion} (see Definition \ref{def:Regular_Expansion}) of any finite set is contained in $\mathcal{T}$. This actually completes the proof, since by Schoenberg's theorem \cite{Sch1938}, every simplex is of this form.
 
\subsection{Almost regular simplices embed into regular polygonal tori}\label{subsec:Almost_regular_simplices_into_tori}

We start with the following easy lemma, which guaranties that every regular simplex can be embedded into a regular polygonal torus.

\begin{lem}\label{lem:Regular_Into_Torus}
Let $\Delta=\{x_i\}_{i\in[n]}$ be a regular simplex. Then for every $m\geq 2$ there exists $r>0$ such that $\Delta$ embeds into the $(m,r)$-regular polygonal torus.
\end{lem}

\begin{proof}
Let $m\geq 2$ and $\Delta=\{x_i\}_{i\in[n]}$ be a regular simplex with side length $\alpha$. We may choose $r>0$ such that the regular $m$-gon $T_{m,r}$ has side length $\alpha/\sqrt{2}$. Let $p$ and $p'$ be two adjacent vertices of $T_{m,r}$. For every $x_i\in \Delta$ let $\widetilde{x}_i=(\widetilde{x}_{ij})_{j\in[n]}\in T_{m,r}^n$ where
\[\widetilde{x}_{ij}	=	\begin{cases}
								p	& \text{ if } j=i\\
								p'	&\text{ otherwise}
					 		\end{cases}
\]
It is straightforward to see that $\norm{\widetilde{x}_i-\widetilde{x}_{i'}}^2=2\norm{p-p'}^2=\alpha^2$ for every $i\neq i'$. Hence, $\widetilde{\Delta}=\{\widetilde{x}_i\}_{i\in[n]}\subset T_{m,r}^n$ is isometric to $\Delta$.
\end{proof} 

The next step is to generalize the above fact by showing that small perturbations of regular simplices also embed into some regular polygonal torus. In what follows, given a matrix $A=(\alpha_{ij})_{i,j\in[n]}$ and a subset $Z=\{z_i\}_{i\in[n]}$ of some Euclidean space, we say that $A$ \emph{is realized from} $Z$ if $\norm{z_i-z_j}=\alpha_{ij}$ for every $i,j\in[n]$. 

The following lemma is a reformulation of a recent result due to Frankl, Pach, Reiher and R\"{o}dl (Lemma 4.9 in \cite{FraPaReRoe2018}).

\begin{lem}\label{lem:Almost_Regular_Matrix}
Let $A=(\alpha_{ij})_{i,j\in[n]}$ be a $n\times n$ real symmetric matrix such that $\alpha_{ii}=0$ for every $i\in[n]$ and $\alpha_{ij}>0$ for every $i,j\in[n]$ with $i\neq j$. Let $\alpha_{\max}=\max_{i,j}\alpha_{ij}$ and suppose that 
\begin{equation}\label{Eq:Almost_Regular_Matrix}
\sum_{1\leq i<j\leq n} (\alpha_{\max}^2-\alpha_{ij}^2)<\alpha_{\max}^2.
\end{equation}
Then there exists a family $\{\Delta_l\}_{l\in[\ell]}$ of regular simplices, where $\ell\leq\binom{n}{2}$, such that $A$ is realized from an affinely independent subset of the product $\prod_{l\in[\ell]}\Delta_l$.
\end{lem}

In order to keep this note self contained we quote below the proof given in \cite{FraPaReRoe2018}.

\begin{proof}
Let $A=(\alpha_{ij})_{i,j\in[n]}$ and let $\alpha_{\max}=\max_{i,j}\alpha_{ij}$. We set
\[
b=\sqrt{\alpha_{\max}^2-\sum_{1\leq i<j\leq n} (\alpha_{\max}^2-\alpha_{ij}^2)}\text{ and } b_{ij}=\sqrt{\alpha_{\max}^2-\alpha_{ij}^2}\text{ for } 1\leq i< j\leq n.
\]

Let $\Delta\subset\mathbb{R}^{n-1}$ be a regular simplex with $n$ vertices and side length $b$. For $1\leq i<j\leq n$ such that $b_{ij}>0$ let $\Delta_{ij}\subset\mathbb{R}^{n-2}$ be a regular simplex with $n-1$ vertices and side length $b_{ij}$. We define
\[
X=\Delta\times\prod_{\substack{i<j\\b_{ij}> 0}}\Delta_{ij}.
\]
Let $\pi:X\to\Delta$ and $\pi_{ij}:X\to\Delta_{ij}$ be the canonical projections. It is not hard to see that we can choose $Z=\{z_i\}_{i=1}^n\subset X$ such that $\pi(Z)=\Delta$, $\pi_{ij}(Z)=\Delta_{ij}$ and $\pi_{ij}(z_i)=\pi_{ij}(z_j)$ for every $z_i,z_j\in Z$ with $i<j$ and $b_{ij}>0$. Notice that for every $s, t\in[n]$ the following holds,
\[\norm{z_{s}-z_t}^2=b^2+\left(\sum_{1\leq i<j\leq n}\!\!\!\!\!b_{ij}^2\right)-b_{s t}^2=\alpha_{st}^2.\]

Finally, since affine dependence of $Z$ implies affine dependence of $\pi(Z)=\Delta$, the set $Z$ must be affinely independent.
\end{proof}

In what follows, a matrix which satisfies the requirements of Lemma \ref{lem:Almost_Regular_Matrix} will be called \emph{almost regular}, while every subset of some Euclidean space that realizes an almost regular matrix will be called \emph{almost regular simplex}. 
 
The next proposition is the extension of Lemma \ref{lem:Regular_Into_Torus} in the case of almost regular simplices. 
 
\begin{prop}\label{prop:Almost_regular_simplex_into_torus}
Let $Z=\{z_i\}_{i\in[n]}$ be an almost regular simplex. Then for every $m\geq 2$ there exists an $m$-regular polygonal torus $T$ such that $Z$ embeds into $T$.  
\end{prop}

\begin{proof}
Let $m\geq 2$. By Lemma \ref{lem:Almost_Regular_Matrix} there exists a family $\{\Delta_l\}_{l\in[\ell]}$ of regular simplices, where $\ell\leq\binom{n}{2}$, such that $Z$ embeds into $\prod_{l\in[\ell]}\Delta_l$. For every $l\in[\ell]$ let $n_l=\abs{\Delta_l}$. By Lemma \ref{lem:Regular_Into_Torus} for every $l\in[\ell]$ there exists $r_l>0$ such that $\Delta_l$ embeds into an $(m,r_l)$-regular polygonal torus of the form $T_{m,r_l}^{n_l}$. Hence, $Z$ embeds into the $m$-regular polygonal torus $T=\prod_{l\in[\ell]}T_{m,r_l}^{n_l}$. 
\end{proof}

\subsection{Every finite set almost embeds into a regular polygonal torus}\label{subsec:Tori_are_dense}

We start with the following definition.  

\begin{defn}\label{def:Almost_Embedding}
Let $\delta>0$, $X\subset \mathbb{R}^k$ and $T\subset \mathbb{R}^{\ell}$. We say that $f:X\to T$ is a \emph{$\delta$-embedding} of $X$ into $T$, if $f$ is an injection and \[\abs*{\norm{f(x)-f(x')}^2-\norm{x-x'}^2}<\delta\] for every $x,x'\in X$. 
\end{defn}
The next lemma formalizes the intuitively obvious fact that every line segment can be approximated with arbitrary accuracy by a large enough circle.  
\begin{lem}\label{lem:tori_dense_one_dim}
 Let  $\delta>0$ and  $X\subset \mathbb{R}$ with $|X|\geq 2$.  Let  $n_0\in\mathbb{N}$ be such that 
\begin{equation}\label{EQ1}
n_0^{-1}\leq  |x-x'|\leq   n_0 \  \text{ for every } x,x'\in X \ \text{ with } x\neq x'.
\end{equation}
Then for every  $n\geq 2\pi n_0^3\delta^{-1}$ the set $X$ is $\delta$-embeddable into a  regular  polygon  $T_{m,r}$ with  $m=n^3$ and $\displaystyle r=\frac{n_0n}{2\pi}$.
\end{lem}
\begin{proof}  
Without loss of generality, we may assume that $X\subset [0, n_0]$.  Let $n \geq 2\pi n_0^3\delta^{-1}$ and for every $x\in X$ let $j(x)$ be the unique non negative integer  satisfying
\[
\frac{j(x)}{n^2}\leq \frac{x}{n_0}< \frac{j(x)+1}{n^2}.
\]
By \eqref{EQ1} the correspondence $x\to j(x)$ is a well defined injection from $X$ into $\{0,1,\dots,n^2\}$. 
For every $x\in X$ we set  $\displaystyle y(x)=\frac{j(x)n_0}{n^2}$. Notice that  $y(x)\in [0, n_0]$ and $\displaystyle 0\leq x-y(x)<n_0n^{-2}$. Hence,
\begin{equation}\label{EQ3}
\left|\ |y(x)-y(x')|^2-\left |x-x'\right|^2\right|< \frac{2n_0^2}{n^2}<\frac{\delta}{2},
\end{equation}
for every $x,x'\in X$.

Let   $\displaystyle \left(r, \frac{2\pi j}{m}\right)$, $j=0,1,\dots,m-1$ be a representation of vertices of $T_{m,r}$ in polar coordinates.
Let $m=n^3$, $r=n_0n/2\pi$ and let   $f:X\to T_{m,r}$ defined by 
\[f(x)=\left(r,\ \frac{2\pi j(x)}{m}\right)=\left(r, \frac{ y(x)}{r}\right). \]
Then $f$ is an injection of $X$ into $T_{m,r}$. To finish the proof it remains to show that $f$ is a $\delta$-embedding.  For every $x,x'\in X$ with $x\neq x'$ we have that
\[
\left\|f(x)-f(x')\right\| =2r\sin\left( \frac{|y(x)-y(x')|}{2r}\right).
\]
Noticing that $\displaystyle \left|\sin^2 t-t^2\right|\leq \abs{t^3}$ and setting $\displaystyle t= \frac{|y(x)-y(x')|}{2r}$ we get that 
\begin{equation}\label{EQ4}
\left|\left\|f(x)-f(x')\right\|^2-\left|y(x)-y(x')\right|^2\right|\leq \frac{\pi n_0^2}{n}<\frac{\delta}{2}
\end{equation}
By \eqref{EQ3} and \eqref{EQ4} the proof is completed.
\end{proof}

\begin{prop}\label{prop:tori_dense} For every $\delta>0$ and every finite set $X\subset \mathbb{R}^k$ there exists $m\in\mathbb{N}$ and $r>0$ such that $X$ is $\delta$-embeddable into an $(m,r)$-regular polygonal torus $T^k_{m,r}$.
\end{prop}

\begin{proof} 
Let $X\subset \mathbb{R}^k$ and $\delta>0$. Let $\pi_i:\mathbb{R}^k\to \mathbb{R}$, $i\in [k]$  be  the canonical projections. By Lemma \ref{lem:tori_dense_one_dim}, we may choose a common pair $(m,r)\in\mathbb{N}\times \mathbb{R}$  such that for every $i\in [k]$ there exists a $\delta/k$-embedding $f_i :\pi_i(X)\to T_{m,r}$.  It is easy to see that the mapping  $f:X\to \mathbb{R}^{2k}$ defined by $\displaystyle f(x)=\left(f_1\left(\pi_1(x)\right), \dots, f_k\left(\pi_k(x)\right)\right)$ is a $\delta$-embedding of $X$ into $T^k_{m,r}$.
\end{proof}

\subsection{Regular expansions of finite sets embed into a regular polygonal torus}

We will need the following definition. 

\begin{defn}\label{def:Regular_Expansion}
Let $X=\{x_i\}_{i\in[n]}\subset\mathbb{R}^k$ and $Y=\{y_i\}_{i\in[n]}\subset \mathbb{R}^d$. We say that $Y$ 
is a regular expansion of $X=\{x_i\}_{i\in[n]}$, if there exists $\alpha>0$ such that \[\norm{y_i-y_j}^2=\norm{x_i-x_j}^2+\alpha^2\] for every $i,j\in[n]$ with $i\neq j$. 
\end{defn}
It is easy to see that a set $Y=\{y_i\}_{i\in[n]}$ is a regular expansion of $X=\{x_i\}_{i\in [n]}$, if and only if, there exists 
a regular simplex $\Delta=\{z_i\}_{i\in[n]}$ such that the set $Y$ is isometric to the set  $Y'=\{(x_i,z_i)\}_{i\in[n]}\subset X\times \Delta$. In particular, since $Y'$ is affinely independent, every regular expansion of a finite set $X$ is a simplex. It was a crucial point of the proof in \citep{FR90} that this property characterizes all simplices in Euclidean spaces. 
\begin{lem}\label{lem:Schonberg}
Every simplex $Y$ is a regular expansion of some other simplex $X$.
\end{lem}

Lemma  \ref{lem:Schonberg} is an  immediate  consequence (see \citep{FR90} for details) of  Schoenberg's \cite{Sch1938} characterization of the finite metric spaces which embed into Euclidean spaces (for more information on this significant theorem the reader may refer to \cite{Alfakih2018} and \cite{Wells1975}).

In view of Lemma \ref{lem:Schonberg}, the next proposition completes the proof of Theorem \ref{thm:Simplices_Into_Tori}.    

\begin{prop}\label{prop:Plus_Epsilon_Into_Torus}
Let $X=\{x_i\}_{i\in[n]}\subset\mathbb{R}^k$. Then every regular expansion of $X$ embeds into an $m$-regular polygonal torus $T$ for some $m\in\mathbb{N}$.
\end{prop}

\begin{proof}

Let $\alpha>0$, $X=\{x_i\}_{i\in[n]}\subset\mathbb{R}^k$, and $Y=\{y_i\}_{i\in[n]}\subset\mathbb{R}^{d}$ be such that $\norm{y_i-y_j}^2=\norm{x_i-x_j}^2+\alpha^2$ for every $i,j\in[n]$ with $i\neq j$. We set $\delta=\alpha^2/n^2$. By Proposition \ref{prop:tori_dense}, we can find $m\in\mathbb{N}$ and $r>0$ such that $X$ is $\delta$-embeddable into an $(m,r)$-regular polygonal torus $T_{m,r}^k$, that is, there exists an injective function $f:X\to T_{m,r}^k$ such that 
$\abs*{\norm{x_i-x_j}^2-\norm{f(x_i)-f(x_j)}^2}<\delta$ for every $i,j\in[n]$.

We set $\delta_{i,j}=\norm{x_i-x_j}^2-\norm{f(x_i)-f(x_j)}^2$ and we define the matrix $A=(\alpha_{ij})_{i,j\in[n]}$, where $\alpha_{ii}=0$ and $\alpha_{ij}=\sqrt{\delta_{ij}+\alpha^2}$ if $i\neq j$. It is easy to check that $A$ is an almost regular matrix. Hence, by Lemma \ref{lem:Almost_Regular_Matrix}, the matrix $A$ is realized from an almost regular simplex $Z=\{z_i\}_{i\in[n]}$. By Proposition \ref{prop:Almost_regular_simplex_into_torus}, there exists an $m$-regular polygonal torus $T_0$ and an isometric embedding $h:Z\to T_0$. We set $Y'=\{y'_i\}_{i\in[n]}$, where $y_i'=(f(x_i),h(z_i))$ for every $i\in[n]$. Notice that $Y'$ is a subset of the $m$-regular polygonal torus $T=T_{m,r}^k\times T_0$. Moreover, for every $i,j\in[n]$ we have that
\begin{align*}
\norm{y_i-y_j}^2		&=\norm{x_i-x_j}^2+\alpha^2\\
					&=\norm{f(x_i)-f(x_j)}^2+\delta_{ij}+\alpha^2\\
					&=\norm{f(x_i)-f(x_j)}^2+\alpha_{ij}^2\\
					&=\norm{f(x_i)-f(x_j)}^2+\norm{h(z_i)-h(z_j)}^2=\norm{y_i'-y_j'}^2.
\end{align*}
Therefore, $Y$ embeds into the $m$-regular polygonal torus $T$ and the proof is completed.
\end{proof}

%--------------------------Bibliography-----------------------%
\subsection*{Acknowledgment}

This research was supported by the Hellenic Foundation for Research and
Innovation (H.F.R.I.) under the “2nd Call for H.F.R.I. Research Projects
to support Faculty Members \& Researchers” (Project Number: HFRI-FM20-02717).
\bibliographystyle{amsplain}
\bibliography{Simplices_and_Regular_Polygonal_Tori_in_Euclidean_Ramsey_Theory_ArXiv_v3}
\end{document}